\documentclass[12pt]{article}

\usepackage[margin=2.5cm, top=2.5cm, bottom=2.5cm]{geometry}

\usepackage[utf8]{inputenc} % allow utf-8 input
\usepackage[T1]{fontenc}    % use 8-bit T1 fonts
\usepackage{url}            % simple URL typesetting
\usepackage{booktabs}       % professional-quality tables
\usepackage{amsfonts}       % blackboard math symbols
\usepackage{nicefrac}       % compact symbols for 1/2, etc.
\usepackage{microtype}      % microtypography
\usepackage{amsmath,amssymb,amsthm,xcolor}
\usepackage{array}
\usepackage{float}
\usepackage[hidelinks,hypertexnames=false]{hyperref}
\usepackage[capitalize]{cleveref}
\theoremstyle{definition}

\theoremstyle{plain}
\newtheorem{theorem}{Theorem}

\newtheorem{lemma}{Lemma}

\usepackage{graphicx}
\usepackage{subcaption}
\usepackage{mdframed}
\usepackage{lipsum}
\usepackage{wrapfig}
\usepackage{tikz}
\usetikzlibrary{arrows.meta}
\usepackage{enumitem}
\allowdisplaybreaks
\usepackage{epstopdf}
\usepackage{booktabs}
\usepackage{changepage}
\usepackage{multirow}
\usepackage{algorithm}
\usepackage{algorithmic}
\usepackage{mathtools}
\usepackage{mathrsfs}
\usepackage{stmaryrd}

\title{\LARGE Nonsmooth rank-one symmetric matrix factorization landscape}

\begin{document}

\author{\large C\'edric Josz\thanks{\url{cj2638@columbia.edu}, IEOR, Columbia University, New York.} \and Lexiao Lai\thanks{\url{lai.lexiao@hku.hk}, Department of Mathematics, The University of Hong Kong, Hong Kong.}}
\date{}

\maketitle
\vspace*{-5mm}%%
\begin{center}
    \textbf{Abstract}
    \end{center}
    \vspace*{-4mm}
 \begin{adjustwidth}{0.2in}{0.2in}
~~~~ Nonsmooth rank-one symmetric matrix factorization has no spurious second-order stationary points.
\end{adjustwidth} 

\section{Introduction}
\label{intro}

The object of this note is to prove the following result.

\begin{theorem}
\label{thm:l1_rank1_symmetric_mf}
For all $u \in \mathbb{R}^n$, the function
\begin{equation*}
\label{eq:objective}
\begin{array}{rccc}
     f : & \mathbb{R}^n & \longrightarrow & \mathbb{R}  \\
         & x & \longmapsto & \frac{1}{2} \sum\limits_{i,j=1}^n |x_ix_j-u_iu_j|
\end{array} 
\end{equation*}
has no spurious second-order stationary points.
\end{theorem}

Spurious means not a global minimum. The analysis of the landscape of $f$ was initiated six years ago \cite{joszneurips2018} as a limiting case of $\ell_p$ rank-one symmetric matrix factorization. This enabled one to show that $f$ has no spurious strict local minima \cite[Proposition 3.1]{joszneurips2018}. Soon after, its global minima were shown to be sharp \cite[Theorem 8.6]{charisopoulos2021low}. It has regained interest recently \cite{guan2024ell_1}, in particular because its saddle points are not strict \cite{davis2022proximal}. 

In spite of the above works, it remains unclear whether Theorem \ref{thm:l1_rank1_symmetric_mf} is true. While it was claimed that $f$ has no spurious local minima \cite[Proposition 1.1]{joszneurips2018}, the proof is invalid, despite providing useful insights (for e.g., the staircase function in \cite[Lemma 3.1]{joszneurips2018}). Similarly, the proof of Theorem \ref{thm:l1_rank1_symmetric_mf} proposed in \cite[Corollary 3.11]{guan2024ell_1} is incomplete, with one of the cases described as too complicated and tedious to be presentable. We thus take this opportunity to prove Theorem \ref{thm:l1_rank1_symmetric_mf} using arguments developed to analyze $\ell_1$ rank-one matrix rectangular factorization \cite{josz2022nonsmooth}.

\section{Proof of Theorem \ref{thm:l1_rank1_symmetric_mf}}

We adhere to the standard notations of Rockafellar and Wets \cite[Chapter 10]{rockafellar2009variational}. Since $f$ is the composition of convex and smooth functions, by the basic chain rule \cite[Theorem 10.6]{rockafellar2009variational} it is regular and
\begin{equation}
\label{eq:subdif}
\partial f(x) = \left\{ \Lambda x ~:~ \Lambda \in \mathrm{sign}(xx^T-uu^T),~ \Lambda^T = \Lambda \right\} 
\end{equation}
where $\mathrm{sign}(t) := t/|t|$ if $t\neq 0$, otherwise $\mathrm{sign}(t) := \left[-1,1\right]$. The partial subdifferential \cite[Corollary 10.11]{rockafellar2009variational} is given by
\begin{subequations}
        \begin{align}
            \partial_{x_i} f(x) & = \sum_{j=1}^n \mathrm{sign}(x_ix_j-u_iu_j)x_j \\
        & = \left\{
    \begin{array}{cc}
     \sum\limits_{j=1}^n \mathrm{sign}(u_i) \mathrm{sign}(x_j(x_i/u_i)-u_j)x_j  & \text{if}~ u_i\neq 0, \\[1mm]
        \mathrm{sign}(x_i) \sum\limits_{j=1}^n |x_j|  & \text{if}~ u_i = 0, 
    \end{array} \right. \\
        & = \left\{
    \begin{array}{cc}
        \mathrm{sign}(u_i) \partial \alpha (x_i/u_i)  & \text{if}~ u_i\neq 0, \\[1mm]
        \mathrm{sign}(x_i) |x|_1  & \text{if}~ u_i = 0, 
    \end{array} \right. \label{eq:partial_sub}
        \end{align}
\end{subequations}
where $|\cdot|_1$ is the $\ell_1$-norm,
\begin{equation}
\label{eq:sub_alpha}
\begin{array}{rccc}
     \alpha : & \mathbb{R} & \longrightarrow & \mathbb{R}  \\
         & t & \longmapsto & \sum\limits_{i=1}^n |x_i t-u_i|
\end{array} ~~~~~ \text{and} ~~~~~ \partial \alpha(t) = \sum_{i=1}^n  \mathrm{sign}(x_it-u_i)x_i.
\end{equation}         
More succinctly, $\alpha(t) = |xt-u|_1$ and $\partial \alpha(t) = \langle \mathrm{sign}(xt-u),x\rangle$ where $\langle \cdot,\cdot\rangle$ is the dot product. Since $\alpha$ is convex and piecewise affine, its subdifferential $\partial \alpha$ is an increasing step function. From the expression of $\partial \alpha(t)$ in \eqref{eq:sub_alpha}, it follows that the jumps between the steps of $\partial \alpha$ occur at $u_i/x_i$ for all $x_i\neq 0$. We call those points the jump points of $\partial \alpha$. Theorem \ref{thm:l1_rank1_symmetric_mf} is proven using three lemmas.

\begin{lemma}
\label{lemma:step}
\normalfont
$0 \in \partial f(x) ~~~ \Longrightarrow ~~~ f(x) = 0 ~~\text{or}~~ 0 \in \partial \alpha(0)$.
\end{lemma}

\begin{proof}
Consider the special case where $u_i \neq 0$ for all $i$. We reason by contradiction: assume that $0 \in \partial f(x)$, $f(x)>0$ and $0 \notin \partial \alpha(0)$. By \cite[Corollary 10.11]{rockafellar2009variational}, we have $0 \in \partial f(x) \subset \partial_{x_1} f(x) \times \cdots \times \partial_{x_n} f(x)$. From the expression of $\partial_{x_i} f(x)$ in \eqref{eq:partial_sub}, we see $0 \in \partial \alpha (x_i/u_i)$ for all $i$. In other words, the ratios $x_i/u_i$ are roots of $\partial \alpha$. Since $0 \notin \partial \alpha(0)$, we have $x_i \neq 0$ for all $i$. If the increasing function $\partial \alpha$ has a positive and a negative root, then $0 \in \partial \alpha (0)$, a contradiction. Thus, without loss of generality, we may assume that $0 <x_1/u_1\leqslant x_2/u_2 \leqslant \cdots \leqslant x_m/u_m$. If $\partial \alpha$ has no positive jump point that is less than or equal to  $x_1/u_1$, then $0 \in \partial \alpha (0)$, a contradiction. Thus let $u_{i_0}/x_{i_0}$ be such a jump point. As prescribed, we have $0<u_{i_0}/x_{i_0}\leqslant x_i/u_i$ for all $i$, as illustrated in Figure \ref{proof_a}. Taking the inverse yields $0 < u_i/x_i \leqslant x_{i_0}/u_{i_0}$, that is to say, all the jump points of $\partial \alpha$ are less than or equal to one of its roots. This is illustrated in Figure \ref{proof_b}.

\begin{figure}[ht!]
%\vspace*{-5mm}
\centering
\begin{subfigure}{.49\textwidth}
  \centering
  \begin{tikzpicture}[scale=0.8]
\draw[->] (-1,2)--(5,2) node[right]{$t$};
\draw[->] (0,-.5)--(0,4.5) node[above]{$\partial \alpha(t)$};
\draw[blue,thick] (-1,0)--(1.2675,0) ;
\draw[blue,thick] (1.25,0)--(1.25,2) ;
\draw[blue,thick] (1.2325,2)--(4.0175,2) ;
\draw[blue,thick] (4,2)--(4,4);
\draw[blue,thick] (3.9825,4)--(5,4) ;
\draw (1.25,2.47) node {$\frac{u_{i_0}}{x_{i_0}}$};
\draw (2.2,1.9)--(2.2,2.1);
\draw (2.2,2.47) node {$\frac{x_1}{u_1}$};
\draw (2.85,2.35);
\draw (2.85,2.47) node {$\hdots$};
\draw (3.5,1.9)--(3.5,2.1);
\draw (3.5,2.47) node {$\frac{x_m}{u_m}$};
\end{tikzpicture}
  \caption{Roots of $\partial \alpha$}
  \label{proof_a}
\end{subfigure}
\begin{subfigure}{.49\textwidth}
  \centering
   \begin{tikzpicture}[scale=0.8]
\draw[->] (-1,2)--(5,2) node[right]{$t$};
\draw[->] (0,-.5)--(0,4.5) node[above]{$\partial \alpha(t)$};
\draw[blue,thick] (-1,-.2)--(1.0175,-.2) ;
\draw[blue,thick] (1,-.2)--(1,0.4) ;
\draw[blue,thick] (0.9825,0.4)--(1.7675,0.4) ;
\draw[blue,thick] (1.75,0.4)--(1.75,.8) ;
\draw[blue,thick] (1.7325,.8)--(2.3175,.8) ;
\draw[blue,thick] (2.3,.8)--(2.3,1.3) ;
\draw[blue,thick] (2.2825,1.3)--(2.6175,1.3) ;
\draw[blue,thick] (2.6,1.3)--(2.6,2) ;
\draw[blue,thick] (2.5825,2)--(4.9825,2) ;
\draw (3.65,1.9)--(3.65,2.1);
\draw (3.65,2.47)  node {$\frac{x_{i_0}}{u_{i_0}}$};
\draw (1,1.9)--(1,2.1);
\draw (1,2.47) node {$\frac{u_m}{x_m}$};
\draw (1.8,2.47) node {$\hdots$};
\draw (2.6,2.47) node {$\frac{u_1}{x_1}$};
\end{tikzpicture}
  \caption{Jump points of $\partial \alpha$}
  \label{proof_b}
\end{subfigure}
\caption{Subdifferential of $\alpha$}
\vspace*{-3mm}
\end{figure}

Consider the case that $\partial \alpha (x_{i_0}/u_{i_0}) \subset (-\infty,0]$ (see Figure \protect\ref{proof_b}). Since $x_{i_0}/u_{i_0}$ is greater than or equal to all of the jump points of $\partial \alpha$, we have $\partial \alpha(t) = 0$ for all $t > x_{i_0}/u_{i_0}$.  Hence for all $t$ large enough we have $\mathrm{sign}(x_it-u_i) = \mathrm{sign}(x_it)$ and $\partial \alpha (t) = \sum_{i=1}^n \mathrm{sign}(x_it)x_i = \sum_{i=1}^n |x_i|$. It follows that $x=0$, a contradiction.

As a result, there exists $\epsilon>0$ such that $[0,\epsilon] \subset \partial \alpha (x_{i_0}/u_{i_0})$, and in particular, $x_{i_0}/u_{i_0}$ is a jump point of $\partial \alpha$. As all the jump points of $\partial \alpha$ (i.e. $u_i/x_i$'s) are less than or equal to $x_{i_0}/u_{i_0}$, we have $x_{i_0}/u_{i_0} = u_1/x_1$, which is the largest jump point. In addition, since $\partial \alpha$ is an increasing step function, it has no roots greater than $x_{i_0}/u_{i_0}$. Therefore, $x_{i_0}/u_{i_0} = x_{m}/u_{m}$.

We next consider the case where $[-\epsilon,\epsilon] \subset \partial \alpha (x_{i_0}/u_{i_0}) = \partial \alpha (x_{m}/u_{m})$ (see Figure \ref{proof_c}), after possibly reducing $\epsilon>0$. Since $\partial \alpha$ is a increasing step function, it has no roots less than $x_{m}/u_{m}$. It follows that $x_i/u_i = x_{m}/u_{m}>0$ for all $i$. As $u_1/x_1 = x_{i_0}/u_{i_0} = x_m/u_m$, we have $x_1/u_1 = \cdots = x_m/u_m = 1$. Hence $f(x) = 0$, a contradiction.
\begin{figure}[ht!]
\centering
\begin{subfigure}{.49\textwidth}
  \centering
  \begin{tikzpicture}[scale=0.8]
\draw[->] (-1,2)--(5,2) node[right]{$t$};
\draw[->] (0,-.5)--(0,4.5) node[above]{$\partial \alpha(t)$};
\draw[blue,thick] (-1,0)--(2.2,0) ;
\draw[blue,thick] (2.2,0)--(2.2,4) ;
\draw[blue,thick] (2.2,4)--(5,4);
\draw (1.9,2.4) node {$\frac{x_{i_0}}{u_{i_0}}$};
\end{tikzpicture}
  \caption{$[-\epsilon,\epsilon]\subset \partial \alpha(x_{i_0}/u_{i_0})$}
  \label{proof_c}
\end{subfigure}
\begin{subfigure}{.49\textwidth}
  \centering
   \begin{tikzpicture}[scale=0.8]
\draw[->] (-1,2)--(5,2) node[right]{$t$};
\draw[->] (0,-.5)--(0,4.5) node[above]{$\partial \alpha(t)$};
\draw[blue,thick] (-1,0)--(1.2675,0) ;
\draw[blue,thick] (1.25,0)--(1.25,2) ;
\draw[blue,thick] (1.2325,2)--(4.0175,2) ;
\draw[blue,thick] (4,2)--(4,4);
\draw[blue,thick] (3.9825,4)--(5,4) ;
\draw (1.25,2.47)  node {$\frac{x_1}{u_1}$};
\draw (3.5,2.47) node {$\frac{x_{i_0}}{u_{i_0}}$};
\end{tikzpicture}
  \caption{$\partial \alpha(x_{i_0}/u_{i_0})= [0,\epsilon]$}
  \label{proof_d}
\end{subfigure}
\caption{Visualization of the cases where $[0,\epsilon] \subset \partial \alpha(x_{i_0}/u_{i_0})$.}
\vspace*{-3mm}
\end{figure}
It remains to consider the case where $\partial \alpha (x_{m}/u_{m}) =\partial \alpha (x_{i_0}/u_{i_0}) = [0,\epsilon]$ (see Figure \ref{proof_d}), possibly after increasing $\epsilon>0$. If $x_{m}/u_{m}\leqslant 1$, then all the jump points $u_i/x_i \geqslant u_{m}/x_{m} \geqslant 1$. This implies that $0$ is a root of $\partial \alpha$, which is a contradiction. Thus $x_{m}/u_{m}> 1 > u_m/x_m = x_1/u_1$. We next prove that $x_i/u_i = x_m/u_m=:\mu$ or $x_i/u_i = x_1/u_1 = 1/\mu$ for every $i$. Assume the contrary that $x_1/u_1<x_i/u_i<x_m/u_m$ for some $i$, then inverting yields that $u_m/x_m<u_i/x_i<u_1/x_1$. This is impossible as $u_m/x_m$ and $u_1/x_1$ are roots (they equal to $x_1/u_m$ and $x_m/u_m$ respectively), and $u_i/x_i$ is a jump point of $\partial \alpha$.  Given $x_i / u_i = \mu$ or $x_i /u_i = 1/\mu$ for every $i$, let $h \in \mathbb{R}^n$ be such that
\begin{equation}
\label{h}
    h_i := \left\{
    \begin{array}{cl}
        -u_i \mu & \text{if}~ x_i/u_i = \mu, \\
        \hphantom{-}u_i/\mu & \text{if}~ x_i/u_i = 1/\mu.
    \end{array}
    \right.
\end{equation}
Consider the function $\gamma:\mathbb{R}^{n\times n}\rightarrow \mathbb{R}$ defined by $\gamma(Q) := h^T Q x + x^T Q h$. Since $0 \in \partial f(x)$, there exists $\Lambda \in \mathrm{sign}(xx^T-uu^T)$ such that $\Lambda x = 0$ and $\Lambda^T = \Lambda$. Thus $\gamma(\Lambda) = h^T \Lambda x + x^T \Lambda h = h^T (\Lambda x) + (\Lambda^T x)^T h = 0$. Yet, observe that 
\begin{subequations}
        \begin{align}
       \gamma\left(\mathrm{sign}(xx^T-uu^T)\right) = & \sum_{i=1}^n \sum_{j=1}^n \mathrm{sign}(x_ix_j-u_iu_j)(h_ix_j + x_ih_j) \label{gamma-1} ~~~~~~~~~~~~~~~~~ \\
        = & \sum\limits_{\frac{x_i}{u_i} = \mu} \sum\limits_{\frac{x_j}{u_j} = \mu} \mathrm{sign}(x_ix_j-u_iu_j)(-u_i \mu \times x_j - x_i\times u_j \mu) ~ + \label{hb}\\
        & \sum\limits_{\frac{x_i}{u_i} = \mu} \sum\limits_{\frac{x_j}{u_j} = \frac{1}{\mu}} \mathrm{sign}(x_ix_j-u_iu_j)(-u_i \mu \times x_j + x_i \times u_j/\mu) ~ + \label{hc}\\
        & \sum\limits_{\frac{x_i}{u_i} =\frac{1}{\mu}} \sum\limits_{\frac{x_j}{u_j} = \mu} \mathrm{sign}(x_ix_j-u_iu_j)(u_i/\mu \times x_j - x_i \times u_j\mu) ~ + \label{hd}\\
        & \sum\limits_{\frac{x_i}{u_i} =\frac{1}{\mu}} \sum\limits_{\frac{x_j}{u_j} =\frac{1}{\mu}} \mathrm{sign}(x_ix_j-u_iu_j)(u_i/\mu \times x_j + x_i \times u_j/\mu)\label{he} \\
        = &-\sum\limits_{\frac{x_i}{u_i} = \mu} \sum\limits_{\frac{x_j}{u_j} = \mu} 2\mu^2~\mathrm{sign}((\mu^2-1)u_iu_j)u_iu_j ~ + \label{xy1-1} \\
        & \sum\limits_{\frac{x_i}{u_i} =\frac{1}{\mu}} \sum\limits_{\frac{x_j}{u_j} =\frac{1}{\mu}}2/\mu^2~\mathrm{sign}((1/\mu^2 - 1)u_iu_j)u_iu_j \label{xy2-1}\\
        = & - 2\mu^2 \sum\limits_{\frac{x_i}{u_i} = \mu} \sum\limits_{\frac{x_j}{u_j} = \mu} |u_i u_j| -2/\mu^2 \sum\limits_{\frac{x_i}{u_i} =\frac{1}{\mu}} \sum\limits_{\frac{x_j}{u_j} = \frac{1}{\mu}} |u_iu_j| ~<~ 0. \label{final-1}
        \end{align}
\end{subequations}
Above, \eqref{gamma-1} follows from the definition of $\gamma$. We substitute $h_i$ using its definition in \eqref{h}, which yields \eqref{hb}-\eqref{he}. We next substitute $x_i$ and $x_j$ using their expressions below the summation signs and obtain \eqref{xy1-1}-\eqref{xy2-1}. Two of the four terms cancel out: \eqref{hb} yields \eqref{xy1-1};  \eqref{hc} cancels out because $-u_i \mu \times x_j + x_i \times u_j/\mu = -u_i \mu \times u_j/\mu + u_i \mu \times u_j/\mu = 0$; \eqref{hd} cancels out because $u_i/\mu \times x_j - x_i \times u_j\mu = u_i/\mu \times u_j \mu - u_i/\mu \times u_j\mu = 0$; \eqref{he} yields \eqref{xy2-1}. To get from \eqref{xy1-1}-\eqref{xy2-1} to \eqref{final-1}, we use the fact that $\mu = x_{m}/u_{m}>1$. We also use the fact that $\mathrm{sign}(u_iu_j)u_iu_j = |u_iu_j|$. The result in \eqref{final-1} is negative because the summation takes place over nonempty sets: $x_{m}/u_{m} = \mu$ and $x_{1}/u_{1} = 1/\mu$. In particular, $\gamma(\Lambda) < 0$ whereas we had shown above that $\gamma(\Lambda) = 0$, a contradiction. 

Consider the general case where $u_i \in \mathbb{R}$ for all $i$. Assume that $0 \in \partial f(x)$. From the expression of $\partial_{x_i} f(x)$ in \eqref{eq:partial_sub}, it follows that $x_i = 0$ if $u_i = 0$. From the expression of $\partial f$ in \eqref{eq:subdif}, there exists $\Lambda \in \mathrm{sign}(xx^T-uu^T)$ such that $\Lambda^T = \Lambda$ and $\Lambda x = 0$. Let $\bar{\Lambda} = (\Lambda_{ij})_{u_iu_j\neq 0}$, $\bar{x} = (x_i)_{u_i \neq 0}$ and $\bar{u} = (u_i)_{u_i \neq 0}$. We have $\bar{\Lambda} \in \mathrm{sign}(\bar{x}\bar{x}^T-\bar{u}\bar{u}^T)$, $\bar{\Lambda}^T = \bar{\Lambda}$ and $\bar{\Lambda} \bar{x} = 0$. Thus $0 \in \partial \bar{f}(\bar{x})$ where $\bar{f}(\bar{y}) = |\bar{y}\bar{y}^T-\bar{u}\bar{u}^T|_1$ for all $\bar{y} \in \mathbb{R}^{|u_i|_0}$ where $|\cdot|_0$ is the $\ell_0$-norm. By the special case, we have $0 = \bar{f}(\bar{x}) = |xx^T-uu^T|_1 = f(x)$ or $0 \in \partial \bar{\alpha}(0) = \partial \alpha (0)$ where $\bar{\alpha}(t) = |\bar{x}t - \bar{u}|_1 = |xt - u|_1 = \alpha(t)$ for all $t\in \mathbb{R}$.
\end{proof}
The next lemma characterize the set of first-order stationary points of $f$.
\begin{lemma}
\label{lemma:stationary}
$(\partial f)^{-1}(0) = \{ x \in \mathbb{R}^n : \langle \mathrm{sign}(u),x\rangle  = 0 , ~|x_i| \leqslant |u_i| , ~ i =1 , \hdots, n \} \cup \{\pm u\}$.
\end{lemma}
\begin{proof} ($\Longrightarrow$)
Assume that $0 \in \partial f(x)$. If $f(x) = 0$, then $xx^T = uu^T$ and so $x = \pm u$. Otherwise, Lemma \ref{lemma:step} and the expression of $\partial \alpha (t)$ in \eqref{eq:sub_alpha} yield
\begin{equation*}
        0 \in \partial \alpha(0) = -\sum_{i=1}^n  \mathrm{sign}(u_i)x_i = -\sum_{u_i\neq 0} \mathrm{sign}(u_i)x_i - \sum_{u_i = 0} \mathrm{sign}(0)x_i.
\end{equation*}
Since $0 \in \partial_{x_i} f(x)$, by \eqref{eq:partial_sub} we have $x_i = 0$ if $u_i = 0$, and so the above equation yields $\langle \mathrm{sign}(u),x\rangle = 0$. Also, by \eqref{eq:partial_sub} we have $0 \in \partial \alpha (x_i/u_i)$ if $u_i\neq 0$. Observe that the increasing function $\partial \alpha$ cannot contain a jump point between the root $0$ and any root $x_i/u_i$. Hence, for all $x_j \neq 0$, if the jump point $u_j/x_j$ is positive, then it is greater than or equal to all the roots $x_i/u_i$, that is to say, $u_j/x_j \geqslant x_i/u_i$. If the jump point $u_j/x_j$ is negative, then it is less than or equal to all the roots $x_i/u_i$, that is to say, $u_j/x_j \leqslant x_i/u_i$. Multiplying both inequalities by $x_j/u_j$ yields $x_ix_j/(u_iu_j) \leqslant 1$ whenever $u_iu_j \neq 0$. In particular, $x_i^2 \leqslant u_i^2$ for all $u_i\neq 0$, and we already know that $x_i = 0$ if $u_i = 0$, so that $|x_i|\leqslant |u_i|$ for all $i$.

($\Longleftarrow$) If $x = \pm u$, then $x$ is a global minimum of $f$ and $0 \in \partial f(x)$ by Fermat's rule \cite[Theorem 10.1]{rockafellar2009variational}. Otherwise, define $\mathrm{sgn}(t) := t/|t|$ if $t\neq 0$, otherwise $\mathrm{sgn}(t) := 0$. Let $\Lambda := -\mathrm{sgn}(uu^T)$, which is symmetric. Since $|x_i| \leqslant |u_i|$ for all $i$, we have 
\begin{equation*}
    \Lambda x =  -\mathrm{sgn}(uu^T)x = -\mathrm{sgn}(u)\mathrm{sgn}(u^T)x =  -\mathrm{sgn}(u) \langle\mathrm{sgn}(u),x\rangle =  -\mathrm{sgn}(u) \langle\mathrm{sign}(u),x\rangle = 0.
\end{equation*}
It remains to show that $\Lambda \in \mathrm{sign}(xx^T-uu^T)$. If $u_iu_j=0$, then $x_ix_j = 0$. In that case, $-\mathrm{sgn}(u_iu_j) = 0$ and $\mathrm{sign}(x_ix_j-u_iu_j) = [-1,1]$. Hence $-\mathrm{sgn}(u_iu_j) \in \mathrm{sign}(x_ix_j-u_iu_j)$. If $u_iu_j \neq 0$, then $x_ix_j/(u_iu_j) \leqslant |x_ix_j/(u_iu_j)| = |x_i|/|u_i| \hspace*{.3mm} |x_j|/|u_j| \leqslant 1$. In that case, $\mathrm{sign}(x_ix_j-u_iu_j) = \mathrm{sign}(u_iu_j)\mathrm{sign}(x_ix_j/(u_iu_j)-1) \ni -\mathrm{sign}(u_iu_j)  = -\mathrm{sgn}(u_iu_j)$.
\end{proof}

Recall that the second subderivative of $f$ at $x$ for $v\in \mathbb{R}^n$ and $w\in \mathbb{R}^n$ is
\begin{equation*}
    d^2 f(x|v)(w):= \liminf_{\tau \searrow 0,w'\rightarrow w} \frac{f(x+\tau w') - f(x) - \tau \langle v,w'\rangle}{\frac{1}{2}\tau^2},
\end{equation*}
as defined in \protect\cite[Definition 13.3]{rockafellar2009variational}. We are now ready to prove our final lemma, which implies Theorem \ref{thm:l1_rank1_symmetric_mf}.
\begin{lemma}
\label{lemma:spurious}
\normalfont
$0 \in \partial f(x) ~~ \text{and} ~~ \forall w \in \mathbb{R}^n, ~ d^2f(x|0)(w) \geqslant 0 ~~~ \Longrightarrow ~~~ f(x)=0$.
\end{lemma}
\begin{proof}
Let $x\in \mathbb{R}^n$ be such that $0 \in \partial f(x)$ and $d^2f(x|0)(w) \geqslant 0$ for all $w\in \mathbb{R}^n$. Since $|x_i| \leqslant |u_i|$ for all $i$, we have $x_ix_j/(u_iu_j) \leqslant 1$ for all $u_iu_j\neq 0$. Let $\theta \in \{\pm 1\}$ and $w := \theta u -x$. When taking a small step $t>0$ in the direction $w$, the inequalities remain valid:
\begin{equation}
\label{eq:ratio}
    \frac{(x_i+t w_i)(x_j+ t w_j)}{u_iu_j} \leqslant 1,~~~\text{if}~~u_iu_j \neq 0.
\end{equation}
Indeed, if $x_ix_j/(u_iu_j) < 1$, then \eqref{eq:ratio} holds by continuity for $t$ small enough. If $x_ix_j/(u_iu_j) = 1$, then $|x_i/u_i|\hspace*{.3mm}|x_j/u_j| = 1$. Since $|x_i/u_i| \leqslant 1$ and $|x_j/u_j| \leqslant 1$, both inequalities must be equalities. Thus $x_i/u_i = x_j/u_j = \pm \theta$. Recall that $w_i = \theta u_i - x_i$ for all $i$. In the case of $+\theta$, we find that $w_i = w_j = 0$ so \eqref{eq:ratio} holds. In the case of $-\theta$, $x_i + t w_i = -\theta u_i + 2t \theta u_i = \theta u_i(2t -1)$ and $x_j + t w_j = \theta u_j(2t -1)$, so that $(x_i + t w_i)(x_j + t w_j)/(u_iu_j) = (2t-1)^2 <1$. It follows that
\begin{subequations}
\label{eq:brutal}
	\begin{align}
	\label{last_brutal_a} f(x+tw) = & \frac{1}{2} \sum\limits_{i,j = 1}^n |(x_i+t w_i) (x_j+t w_j) - u_i u_j| ~~~~~~~~~~~~~~~~~~~\\
	\label{last_brutal_b} = & \frac{1}{2} \sum\limits_{u_iu_j \neq 0} |(x_i+t w_i) (x_j+t w_j) - u_i u_j|\\
	\label{last_brutal_c} = & \frac{1}{2} \sum\limits_{u_iu_j \neq 0} |u_i u_j|\left|\frac{(x_i+t w_i) (x_j+t w_j)}{u_iu_j} - 1\right|\\
	\label{last_brutal_d} = & \frac{1}{2} \sum\limits_{u_iu_j \neq 0} |u_i u_j|\left(1 - \frac{(x_i+t w_i) (x_j+t w_j)}{u_iu_j}\right)\\
	\label{last_brutal_e} = & \frac{1}{2} \sum\limits_{u_iu_j \neq 0} |u_i u_j| - \frac{1}{2} \sum\limits_{u_iu_j \neq 0} \mathrm{sign}(u_iu_j)(x_i+t w_i) (x_j+t w_j)\\
	\label{last_brutal_f} = & \frac{1}{2} \sum\limits_{u_iu_j \neq 0} |u_i u_j| - \frac{1}{2} \sum\limits_{u_i \neq 0} \mathrm{sign}(u_i)(x_i+t w_i)\sum\limits_{u_j \neq 0}  \mathrm{sign}(u_j)(x_j+t w_j) \\
	\label{last_brutal_g} = & \frac{1}{2} \sum\limits_{u_iu_j \neq 0} |u_i u_j| - \frac{1}{2} \left((1-t)\sum\limits_{u_i \neq 0} \mathrm{sign}(u_i)x_i+t\sum\limits_{u_i \neq 0}\mathrm{sign}(u_i)u_i \theta \right) \times \\ \label{last_brutal_h} & 
	\left((1-t)\sum\limits_{u_j \neq 0}  \mathrm{sign}(u_j)x_j+t\sum\limits_{u_j \neq 0}\mathrm{sign}(u_j)u_j\theta \right) \\
	\label{last_brutal_i} = & \frac{1}{2} \sum\limits_{u_iu_j \neq 0} |u_i u_j| -\frac{ t^2}{2} \sum\limits_{u_i \neq 0} \mathrm{sign}(u_i)u_i  \sum\limits_{u_j \neq 0} \mathrm{sign}(u_j)u_j \\
	\label{last_brutal_j} = & (1-t^2)f(x).
	\end{align}
\end{subequations}
Above, \eqref{last_brutal_b} holds because, according to \eqref{eq:partial_sub}, $x_i=0$ when $u_i=0$, and so $x_i+t w_i = x_i + t (\theta u_i - x_i) = 0$ when $u_i=0$. \eqref{last_brutal_c} is obtained by factorizing each term in the sum by $|u_i u_j|$. \eqref{eq:ratio} implies that the term inside the absolute value is nonpositive, hence \eqref{last_brutal_d}. \eqref{last_brutal_e} is the result of expanding the product inside the sum and the fact that $\mathrm{sign}(a) = |a|/a$ when $a \neq 0$. \eqref{last_brutal_f} is obtained by factorizing the second term in \eqref{last_brutal_e}. \eqref{last_brutal_g}-\eqref{last_brutal_h} uses that $w_i = \theta u_i - x_i$. \eqref{last_brutal_i} is due Lemma \ref{lemma:stationary}, which results in two terms cancelling out. \eqref{last_brutal_j} holds because we are computing $f(x+tw)$, so that evaluated at $t=0$, it must be equal to $f(x)$. It follows from {\protect\eqref{eq:brutal}} that $d^2f(x|0)(w) \leqslant - 2f(x)$. Thus $f(x) \leqslant  -d^2f(x|0)(w)/2 \leqslant 0$. This completes the proof as $f$ is nonnegative.
\end{proof}

\noindent\textbf{Acknowledgements} We thank the reviewer and the editors for their valuable feedback.

\bibliographystyle{abbrv}    
\bibliography{references}

\begin{thebibliography}{1}

\bibitem{charisopoulos2021low}
V.~Charisopoulos, Y.~Chen, D.~Davis, M.~D{\'\i}az, L.~Ding, and D.~Drusvyatskiy.
\newblock Low-rank matrix recovery with composite optimization: good conditioning and rapid convergence.
\newblock {\em Foundations of Computational Mathematics}, pages 1--89, 2021.

\bibitem{davis2022proximal}
D.~Davis and D.~Drusvyatskiy.
\newblock Proximal methods avoid active strict saddles of weakly convex functions.
\newblock {\em Foundations of Computational Mathematics}, 22(2):561--606, 2022.

\bibitem{guan2024ell_1}
J.~Guan and A.~M.-C. So.
\newblock $\ell_1$-norm rank-one symmetric matrix factorization has no spurious second-order stationary points.
\newblock {\em arXiv preprint arXiv:2410.05025}, 2024.

\bibitem{josz2022nonsmooth}
C.~Josz and L.~Lai.
\newblock Nonsmooth rank-one matrix factorization landscape.
\newblock {\em Optimization Letters}, 16(6):1611--1631, 2022.

\bibitem{joszneurips2018}
C.~Josz, Y.~Ouyang, R.~Y. Zhang, J.~Lavaei, and S.~Sojoudi.
\newblock {A theory on the absence of spurious solutions for nonconvex and nonsmooth optimization}.
\newblock {\em NeurIPS}, Dec. 2018.

\bibitem{rockafellar2009variational}
R.~T. Rockafellar and R.~J.-B. Wets.
\newblock {\em Variational analysis}, volume 317.
\newblock {Springer Berlin}, Heidelberg, 2009.

\end{thebibliography}

\end{document}